\documentclass{amsart}

\usepackage{amssymb,amsmath,amscd}

\usepackage[colorlinks,linktocpage]{hyperref}
\hypersetup{linkcolor=[rgb]{0,0,0.715}}
\hypersetup{citecolor=[rgb]{0,0.715,0}}

\newtheorem{thm}{Theorem}[section]
\newtheorem{cor}[thm]{Corollary}
\newtheorem{prop}[thm]{Proposition}
\newtheorem{lem}[thm]{Lemma}

\newtheorem{quest}[thm]{Question}
\newtheorem{mainthm}{Theorem}

\theoremstyle{definition}
\newtheorem{defn}[thm]{Definition}

\newtheorem{exmp}[thm]{Example}

\theoremstyle{remark}

\newtheorem{rems}[thm]{Remarks}

\makeatletter
\let\c@equation\c@thm
\makeatother
\numberwithin{equation}{section}

\newcommand{\intrr}{\mathcal{R}^\circ_{\epsilon}}
\newcommand{\intr}{\mathcal{R}^\circ_{\epsilon,\delta}}
\newcommand{\intrp}{\mathcal{R}^\circ_{\epsilon',\delta'}}

\newcommand{\ghto}{\overset{GH}\longrightarrow}

\newcommand{\Ric}{\mathrm{Ric}}
\newcommand{\Hmeas}{\mathcal{H}^n}
\newcommand{\dimH}{\dim_{\mathcal{H}}}
\newcommand{\susp}{\mathrm{Susp}}

\title[Ricci curvature and fundamental groups of effective regular sets]{Ricci curvature and fundamental groups \\ of effective regular sets}

\author{Jiayin Pan}
\address{Department of Mathematics, University of California, Santa Cruz, CA, USA.}
\email{jpan53@ucsc.edu}

\begin{document}
	
	\maketitle
	\begin{center}
		\textit{In honor of Xiaochun Rong on his 70th birthday}
	\end{center}

    \begin{abstract}
    	For a Gromov-Hausdorff convergent sequence of closed manifolds $M_i^n\ghto X$ with $$\Ric\ge-(n-1),\quad \mathrm{diam}(M_i)\le D,\quad \mathrm{vol}(M_i)\ge v>0,$$
    	we study the relation between $\pi_1(M_i)$ and $X$. It was known before that there is a surjective homomorphism $\phi_i:\pi_1(M_i)\to \pi_1(X)$ by the work of Pan-Wei \cite{PW}. In this paper, we construct a surjective homomorphism from the interior of the effective regular set in $X$ back to $M_i$, that is, $\psi_i:\pi_1(\intr)\to \pi_1(M_i)$. These surjective homomorphisms $\phi_i$ and $\psi_i$ are natural in the sense that their composition $\phi_i \circ \psi_i$ is exactly the homomorphism induced by the inclusion map $\intr \hookrightarrow X$.
    \end{abstract}

	\section{Introduction}\label{sec:intro}
	
	For a Gromov-Hausdorff convergent sequence $M_i\ghto X$ with curvature bounds, it is crucial to understand the relationship between $M_i$ and $X$. For example, when $M_i$ are closed $n$-manifolds with
	$$\sec\ge-1,\quad \mathrm{diam}(M_i)\le D,\quad \mathrm{vol}(M_i)\ge v>0,$$
	Perelman proved that $M_i$ is homeomorphic to $X$ for all $i$ large \cite{Per}. For the context of this paper, let us consider a convergent sequence of closed $n$-manifolds $M_i\ghto X$ with Ricci curvature lower bounds
	\begin{equation}\label{eq:cond}
	\Ric\ge-(n-1),\quad \mathrm{diam}(M_i)\le D,\quad \mathrm{vol}(M_i)\ge v>0
	\end{equation}
	Under this weaker condition, one cannot expect $X$ to be homeomorphic to $M_i$. By the work of Wei and the author \cite{PW}, the limit space $X$ is semi-locally simply connected. This was later generalized to the collapsing case by Wang \cite{Wang}. As a consequence, there is a forward surjective homomorphism from $\pi_1(M_i)$ to $\pi_1(X)$.
	
	\begin{thm}\cite{PW}\label{thm:forward_surj}
	Let $M_i$ be a sequence of closed $n$-manifolds with (\ref{eq:cond})
	and Gromov-Hausdorff converging to a limit space $X$. Let $x_i\in M_i$ be a sequence of points converging to $x\in X$. Then for all $i$ large, there is a surjective homomorphism $$\phi_i:\pi_1(M_i,x_i)\to \pi_1(X,x).$$
	\end{thm}
	
	For an element $[\sigma_i]\in \pi_1(M_i,x_i)$ represented by a loop $\sigma_i$ based at $x_i$, its image under this forward homomorphism $\phi_i$ is constructed by drawing a loop $\sigma$ in $X$ that is sufficiently close to $\sigma_i$; see \cite{Tu,PW}. While $\phi_i$ is surjective, in general it is not injective even under the noncollapsing condition. In fact, there could be shorter and shorter non-contractible loops at $x_i$ with length tending to $0$, then by construction $\phi_i$ sends them to identity. We will review an example by Otsu \cite{Ot} in Section \ref{sec:exmp} regarding this.
	
	From Theorem \ref{thm:forward_surj}, because $\phi_i$ may have a kernel, it appears that some elements in $\pi_1(M_i)$ are lost in the limit $X$. As the main result of this paper, we show that all elements in $\pi_1(M_i)$ are still retained in $X$; more specifically, in the effective regular set $\mathcal{R}_{\epsilon,\delta}$ of $X$. In fact, we will construct a backward surjective homomorphism from $\pi_1(\intr,x)$ to $\pi_1(M_i,x_i)$, where $\intr$ is the interior of $\mathcal{R}_{\epsilon,\delta}$ and $x\in X$ is a regular point. By the regularity theory developed by Cheeger-Colding \cite{CC1}, $\intr$ is a connected topological manifold of dimension $n$ for all $0<\epsilon\le \epsilon(n)$ and $\delta>0$.
	
\begin{mainthm}\label{thm:backward_surj}
	Let 
	$$(M_i,x_i)\ghto (X,x)$$
	be a convergent sequence of closed $n$-manifolds with (\ref{eq:cond}), where $x$ is a regular point. Then\\
	(1) for any $0<\epsilon<\epsilon(n)$ and sufficiently small $0<\delta< \delta(\epsilon,x)$, there is a surjective homomorphism
	$$\psi^\delta_i:\pi_1(\intr,x)\to \pi_1(M_i,x_i)$$ for all $i$ large;\\
	(2) the composition of $\psi^\delta_i$ and $\phi_i$ in Theorem \ref{thm:forward_surj}
	$$\phi_i \circ \psi^\delta_i: \pi_1(\intr,x)\to \pi_1(X,x)$$
	is exactly the homomorphism $\iota_\star$ induced by the inclusion map $\iota: \intr \hookrightarrow X$.
\end{mainthm}	
	
The construction of this backward homomorphism $\psi_i$ is natural and similar to that of $\phi_i$: namely, by drawing nearby loops. The surjectivity of $\psi_i$ requires a complete different and more involved argument than that of $\phi_i$. We remark that $\psi_i$ is not injective in general. In fact, we will review an example by Anderson \cite{An} in Section \ref{sec:exmp}; in this example, both $M_i$ and $X$ are simply connected but $\pi_1(\intr)$ is isomorphic to $\mathbb{Z}_2$.	
	

As an application of Theorem \ref{thm:forward_surj}, we show that if the inclusion map $\intrr \hookrightarrow X$ induces an injective homomorphism 
$\iota_\star: \pi_1(\intrr,x)\to \pi_1(X,x)$, then $\pi_1(M_i)$ is isomorphic to $\pi_1(X)$ for all $i$ large. Note that we are considering the $\epsilon$-regular set in this statement; in other words, the involvement of $\delta$ is dropped. 

\begin{mainthm}\label{thm:reg_inj_isom}
	Let 
	$$(M_i,x_i)\ghto (X,x)$$
	be a convergent sequence of closed $n$-manifolds with (\ref{eq:cond}), where $x$ is a regular point. Suppose that for some $0<\epsilon<\epsilon(n)$, the induced homomorphism 
	$$\iota_\star: \pi_1(\intrr,x)\to \pi_1(X,x)$$
	is injective, then $\pi_1(M_i)$ is isomorphic to $\pi_1(X)$ for all $i$ large. In particular, if $\intrr$ is simply connected, then so is $M_i$.
\end{mainthm}


The work in this paper is motivated by the $\pi_1$-stability problem: 

\begin{quest}\label{quest:pi1_stab}
	Given $n,D,v>0$, is there a positive constant $\epsilon(n,D,v)>0$ such that if two closed $n$-manifolds $M_1$ and $M_2$ satisfy (\ref{eq:cond}) and $d_{GH}(M_1,M_2)\le \epsilon$, then $\pi_1(M_1)$ and $\pi_1(M_2)$ are isomorphic?
\end{quest}

As a comparison, if one replace Ricci curvature in (\ref{eq:cond}) by a sectional curvature lower bound $\mathrm{sec}\ge -1$, then $M_1$ and $M_2$ are homeomorphic when they are Gromov-Hausdorff close; see the works by Grove-Petersen-Wu \cite{GPW} and Perelman \cite{Per}.

Question \ref{quest:pi1_stab} is a stronger version of the celebrated finiteness result by Anderson \cite{An} below. In fact, if Question \ref{quest:pi1_stab} has an affirmative answer, then finiteness would easily follow by a standard contradicting argument. 

\begin{thm}\cite{An}\label{thm:pi1_finiteness}
	Given $n,D,v>0$, there are finitely many isomorphism classes of fundamental groups among closed $n$-manifolds with (\ref{eq:cond}).
\end{thm}

To resolve Question \ref{quest:pi1_stab}, it is equivalent to answer:

\begin{quest}\label{quest:determine_pi1}
	For a convergent sequence of closed $n$-manifolds $M_i\ghto X$ with (\ref{eq:cond}), is it possible to determine $\pi_1(M_i)$ solely from $X$?
\end{quest}

Theorems \ref{thm:backward_surj} and \ref{thm:reg_inj_isom} provide partial answers to Question \ref{quest:determine_pi1}. 

\begin{rems} Let us mention other related results regarding Questions \ref{quest:pi1_stab} and \ref{quest:determine_pi1}.\\
(1) When $X$ satisfies a local half-volume lower bound, we have a positive answer; see \cite[Section 3]{PW}.\\
(2) If one considers the equivariant Gromov-Hausdorff convergence of the Riemannian universal covers, then it holds that $\pi_1(M_i,p_i)$ is isometric to the limit group for all $i$ large (see \cite[Section 2.3]{PR} for details). However, because a subsequence was chosen to derive equivariant convergence, this result does not provide direct answers to Question \ref{quest:pi1_stab}.
\end{rems}

The proofs in this paper relies on several ingredients. The first one is the regularity theory of non-collapsing Ricci limit spaces developed by Cheeger-Colding \cite{Co,CC1,CC2,Ch_book}. The second ingredient is the equivariant convergence under Ricci and volume lower bounds; in particular, we utilize some of the results by Pan-Rong \cite{PR} and Chen-Rong-Xu \cite{CRX}. Lastly, we use some of the methods in Pan-Wei's work \cite{PW} on loops and homotopies under Gromov-Hausdorff convergence; these techniques can be traced back to the work of Borsuk \cite{Bor} and Tuschman \cite{Tu}.

\tableofcontents

\textit{Acknowledgment.} The author is partially supported by National Science Foundation DMS-2304698 and Simons Foundation Travel Support for Mathematicians. The author would like to thank Xiaochun Rong and Guofang Wei for many fruitful discussions through the years. The author would like to thank Dimitri Navarro for suggestions on an early draft of this paper.

\section{Preliminaries}\label{sec:pre}

\subsection{Regularity theory of noncollapsing Ricci limit spaces}

Throughout the paper, we always use $\Psi(\epsilon|n)$ to represent some nonnegative function depending on $\epsilon$ and $n$ with
$$\lim\limits_{\epsilon\to 0} \Psi(\epsilon|n)=0.$$ 
We may use the same symbol $\Psi(\epsilon|n)$ though dependence on $\epsilon$ or $n$ may be different. 

Given $n\in\mathbb{N}$, $\kappa\ge 0$, and $v>0$, we denote $\mathcal{M}(n,-\kappa,v)$ the set of all pointed Ricci limit spaces $(X,x)$ coming from some GH convergent sequence of complete $n$-manifolds $(M_i,p_i)$ with 
\begin{equation}\label{eq:local_cond}
\Ric\ge -(n-1)\kappa,\quad \mathrm{vol}(B_1(p_i))\ge v>0.
\end{equation}
The regularity theory about these noncollapsing Ricci limit spaces are mainly developed by Cheeger, Colding, and Naber. Below, we review some of the results that will be used later. The main references are \cite{CC2,Ch_book}.

\begin{defn}\cite{CC2,Ch_book}
Let $\epsilon,\delta>0$. For a Ricci limit space $X\in \mathcal{M}(n,-1,v)$, we define $(\epsilon,\delta)$-regular set, $\epsilon$-regular set, regular set, and singular set of $X$ as below.
\begin{align*}
\mathcal{R}_{\epsilon,\delta}&=\{x\in X\ |\  d_{GH}(B_r(x),B_r^n(0))\le \epsilon r \text{ for all } 0<r\le \delta\},\\
\mathcal{R}_\epsilon &= \bigcup_{\delta>0} \mathcal{R}_{\epsilon,\delta}, \\
\mathcal{R}&=\bigcap_{\epsilon>0} \mathcal{R}_\epsilon=\bigcap_{\epsilon>0}\bigcup_{\delta>0} \mathcal{R}_{\epsilon,\delta}.\\
\mathcal{S}&=X-\mathcal{R}.
\end{align*}
\end{defn}

\begin{thm}\cite{Co,CC2}\label{pre:CC_vol_conv}
	Let $(M^n_i,p_i)\ghto (X,x)$ be a convergent sequence with (\ref{eq:local_cond}). Then for all $r>0$, we have volume convergence
	$$\mathrm{vol}(B_r(p_i))\to \Hmeas(B_r(x))$$
	as $i\to\infty$, where $\Hmeas$ is the $n$-dimensional Hausdorff measure on $X$.
\end{thm}

\begin{thm}\cite{Co,CC2}\label{pre:CC_max_vol}
	Let $X\in\mathcal{M}(n,-\delta,v)$ and $x\in X$.\\
	(1) If $$d_{GH}(B_1(x),B_1^n(0))\le\delta,$$ then
	$$\Hmeas(B_1(x))\ge (1-\Psi(\delta|n))\mathrm{vol}(B_1^n(0)).$$
	(2) If $$\Hmeas(B_1(x))\ge (1-\delta)\mathrm{vol}(B_1^n(0)),$$ then $$d_{GH}(B_1(x),B_1^n(0))\le\Psi(\delta|n).$$
\end{thm}

The following facts follow from Theorems \ref{pre:CC_vol_conv} and \ref{pre:CC_max_vol}.
\begin{lem}\label{lem:reg_param}
	Let $X\in\mathcal{M}(n,-1,v)$.\\
	(1) Given $\epsilon,\delta>0$, there are $\epsilon'=\Psi(\epsilon|n)$ and $\delta'=\delta/3$ such that $\mathcal{R}_{\epsilon,\delta}\subseteq \intrp$.\\
	(2) Let $A$ be a compact subset of $\mathcal{R}$. Then for any $\epsilon>0$, there is $\delta>0$ such that $A \subseteq \mathcal{R}_{\epsilon,\delta}$.
\end{lem}

\begin{proof}
	We include the proof here for readers' convenience.
	
	(1) Let $x\in\mathcal{R}_{\epsilon,\delta}$. By definition, this means
	$$d_{GH}(B_r(x),B_r^n(0))\le \epsilon r $$
	for all $0<r\le \delta$. By Theorem \ref{pre:CC_max_vol} and Bishop-Gromov relative volume comparison, 
	$$\Hmeas(B_s(y))\ge (1-\Psi(\epsilon|n))\mathrm{vol}(B_s^n(0))$$
	holds for all $y\in B_{\delta/3}(x)$ and all $0<s\le \delta/3$.
	Applying Theorem \ref{pre:CC_max_vol}(2), we see that 
	$$d_{GH}(B_s(y),B_s^n(0))\le \Psi(\epsilon|n)s,$$
	that is, $y\in \mathcal{R}_{\Psi(\epsilon|n),\delta/3}$ for all $y\in B_{\delta/3}(x)$. Therefore, $x\in \intrp$, where $\epsilon'=\Psi(\epsilon|n)$ and $\delta'=\delta/3$.
	
	(2) Let $\epsilon>0$. For each $x\in A$, we pick $\delta(x)>0$ as the largest $\delta$ so that
	$$d_{GH}(B_r(x),B_r^n(0))\le \epsilon r$$
	holds for all $0<r\le\delta$. It suffices to show that $\delta(x)$ has a uniform positive lower bound for all $x\in A$. We argue by contradiction. Suppose that there is a sequence $x_i\in A$ with $\delta(x_i)\to 0$. Then by compactness of $A$, $x_i$ subconverges to some $y\in A$, which is also regular. Therefore, for $\epsilon'>0$, which will be determined later, there is $\delta_0=\delta_0(\epsilon',y)>0$ such that $y\in\mathcal{R}_{\epsilon',\delta_0}$. Thus it follows from Theorem \ref{pre:CC_max_vol}(1) that
	$$\Hmeas(B_{\delta_0}(y))\ge (1-\Psi(\epsilon'|n))\mathrm{vol}(B_{\delta_0}^n(0)).$$
	By volume convergence,
	$$\Hmeas(B_{\delta_0}(x_i))\ge (1-2\Psi(\epsilon'|n))\mathrm{vol}(B_{\delta_0}^n(0))$$
	for $i$ large, thus 
	$$d_{GH}(B_{r}(x_i),B_{r}^n(0))\le \Psi'(\epsilon'|n) r$$
	for all $0<r\le \delta_0$. Now we choose $\epsilon'>0$ so that $\Psi'(\epsilon'|n)\le \epsilon$, then $x_i\in \mathcal{R}_{\epsilon,\delta_0}$ for all $i$ large. A contradiction to $\delta(x_i)\to 0$. This complete the proof.
\end{proof}

\begin{thm}\cite{CC2}\label{pre:CC_sing_dim}
	Let $X\in\mathcal{M}(n,-1,v)$. Then its singular set $\mathcal{S}$ has Hausdorff dimension at most $n-2$. 
\end{thm}

\begin{thm}\cite{CC2}\label{pre:CC_cnt}
	Let $X\in\mathcal{M}(n,-1,v)$ and let $A$ be a closed subset of $X$ with $\mathcal{H}^{n-1}(A)=0$. Then $X-A$ is path connected. Moreover, given any $\delta>0$ and any pair of points $x,y\in X-A$, a path $\sigma$ in $X-A$ between $x,y$ can be chosen that
	$$\mathrm{length}(\sigma) \le (1+\delta) d(x,y).$$
\end{thm}

\begin{thm}\cite{CC2}\label{pre:CC_reg_contract}
	Given dimension $n$, there is a constant $\epsilon_0(n)>0$ such that the following holds for all $0<\epsilon\le \epsilon_0(n)$.
	
	Let $X\in\mathcal{M}(n,-1,v)$ and $x\in X$ such that 
	$$d_{GH}(B_\delta(x),B_\delta^n(0))\le \epsilon \delta,$$
	where $\delta>0$. Then $B_r(x)$ is contractible in $B_{2r}(x)$ for all $0<r\le \delta/10$.
\end{thm}


\subsection{Equivariant GH convergence with Ricci and volume lower bounds}

In the study of fundamental groups associated to a convergent sequence
$$(M_i,x_i)\ghto (X,x)$$
with conditions (\ref{eq:local_cond}), it is natural to take the universal covers and their convergence into account. A powerful tool is the equivariant Gromov-Hausdorff convergence introduced by Fukaya-Yamaguchi \cite{FY}. After passing to a subsequence, we can obtain convergence
\begin{equation}\label{CD:cover}
	\begin{CD}
		(\widetilde{M}_i,\tilde{x}_i,\Gamma_i) @>GH>> 
		(Y,y,\Gamma)\\
		@VV\pi_i V @VV\pi V\\
		(M_i,x_i) @>GH>> (X,x).
	\end{CD}
\end{equation}
Here $\Gamma_i=\pi_1(M_i,x_i)$ acts isometrically, freely, and discretely on the universal cover $(\widetilde{M}_i,\tilde{x}_i)$. This sequence of $\Gamma_i$-actions converges to a limit isometric $\Gamma$-action on the limit space $Y$. Due to the noncollapsing condition on $(M_i,x_i)$, the limit group $\Gamma$ is a discrete subgroup of $\mathrm{Isom}(Y)$; see Corollary \ref{cor:discrete_limit}.

We below state a result by Chen-Rong-Xu \cite{CRX}, which roughly states that if a point  $z\in Y$ is sufficiently regular, then $\Gamma$-action cannot fix $z$.

\begin{thm}\cite[Theorem 2.1 and Corollary 2.2]{CRX}\label{pre:no_fix_sing}
	Given $n,v>0$, there is a constant $\epsilon(n,v)>0$ such that the following holds.
	
	In the convergence (\ref{CD:cover}) with conditions (\ref{eq:local_cond}), if $z\in Y$ is $(\epsilon,\delta)$-regular, where $\delta>0$, then $\Gamma$ acts freely on $B_{\delta/4}(z)$.
\end{thm}

We will also need a quantitative result describing the action of any non-trivial subgroup of $\mathrm{Isom}(Y)$, which is proved in a joint work by Rong and the author \cite{PR}. Given a subgroup $H\leq \mathrm{Isom}(Y)$, we write its displacement on a $1$-ball by
$$D_{1,y}(H)=\sup \{d(hz,z) | z\in B_1(y),h\in H \}.$$

\begin{thm}\cite[Theorem 0.8]{PR}\label{pre:nss}
	Given $n,v>0$, there is a constant $\delta(n,v)>0$ such that for any space $(Y,y)\in\mathcal{M}(n,-1,v)$ and any nontrivial subgroup of $H$ of $\mathrm{Isom}(X)$, $D_{1,y}(H)\ge\delta(n,v)$ holds.
\end{thm}

\section{Illustrative examples}\label{sec:exmp}

In this short section, we briefly review some relevant examples of convergent sequences $M_i\overset{GH}\to X$ with conditions (\ref{eq:cond}) by Otsu \cite{Ot} and Anderson \cite{An}. In particular, we shall see that in general the homomorphisms $\phi_i$ in Theorem \ref{thm:forward_surj} and $\psi_i$ in Theorem \ref{thm:backward_surj} are not injective.

\begin{exmp}\label{exmp:otsu}
Otsu \cite{Ot} constructed a sequence of doubly warped metric products on $M=S^{p+1}\times S^q$, where $p\ge 2$ and $q\ge 2$:
$$[0,b_i] \times_{f_i} S^p \times_{h_i} S^q, \quad g_i= dr^2 + f_i^2(r)ds_p^2 + h_i^2(r) ds_q^2.$$
such that
$$\Ric(g_i)\ge n-1,\quad \mathrm{diam}(g_i)=b_i\to \pi,\quad \mathrm{vol}(g_i)\ge v>0.$$
At $s=0$ or $b_i$, $f_i$ and $g_i$ satisfies
$$f_i(s)=0,\quad f'_i(s)=1,\quad h_i(s)>0,\quad \lim\limits_{i\to\infty} h_i(s)\to 0  \quad h'_i(s)=0.$$
As $i\to\infty$, $(M,g_i)$ converges to $\susp(S^p \times S^q)$, a suspension over $S^p\times S^q$.

Since $S^q$-factor is always the round sphere in the construction, we can take the antipodal $\mathbb{Z}_2$-action on $S^q$ factor and consider the quotient $(N_i,\bar{g}_i)=(M,g_i)/\mathbb{Z}_2.$ The resulting $(N_i,\bar{g}_i)$ is Riemannian because $\mathbb{Z}_2$-action is isometric and free on $(M,g_i)$.
Then as $i\to\infty$, $N_i$ converges to $X=\mathrm{Susp}(S^p \times \mathbb{R}P^q)$. In terms of fundamental groups, we have
$$\pi_1(N_i)=\mathbb{Z}_2,\quad \pi_1(X)=\mathrm{id}.$$
The forward homomorphism $\phi_i:\pi_1(N_i) \to \pi_1(X)$ has kernel $\mathbb{Z}_2$. The limit space $X$ has two singular points as the vertices of the suspension. For small $\epsilon$ and $\delta>0$, $\intr$ is homeomorphic to 
$(0,1)\times S^p \times \mathbb{R}P^q$. In particular, $\pi_1(\intr)=\mathbb{Z}_2$.

\end{exmp}

\begin{exmp}\label{exmp:EH}
	Modifying the Eguchi-Hanson metric \cite{EH} on $TS^2$, the tangent bundle of $S^2$, Anderson \cite{An} constructed a sequence of metrics $g_i$ on $M^4$, the double of the disk bundle in $TS^2$, with
	$$\Ric(g_i)\ge 0,\quad \mathrm{diam}(g_i)\le D,\quad  \mathrm{vol}(g_i)\ge v>0.$$
	$M$ is diffeomorphic to $S^2\times S^2$. Recall that the Eguchi-Hanson metric, written as $h$, on $TS^2$ is Ricci-flat and has Euclidean volume growth. It has a unique asymptotic cone as $C(\mathbb{R}P^3)=\mathbb{R}^4/\mathbb{Z}_2$. 
	
	Let $\mathcal{Z}$ be the zero-section in $TS^2$ and let $B_i=T_1(\mathcal{Z},r_i^{-2}h)$ be the tubular neighborhood of $\mathcal{Z}$ of radius $1$ with respect to the metric $r_i^{-2}h$, where $r_i\to\infty$. Modifying the metric around $\partial B_i$ and then doubling it, one obtains the desired metric $g_i$ on $M$. As $i\to\infty$, $(M,g_i)$ converges to $X=\susp(\mathbb{R}P^3)$, a suspension over $\mathbb{R}P^3$. $X$ has two singular points as the vertices. For small $\epsilon,\delta>0$, $\intr$ is homeomorphic to $(0,1)\times \mathbb{R}P^3$. Hence
	$$\pi_1(M)=\pi_1(X)=\mathrm{id},\quad \pi_1(\intr)=\mathbb{Z}_2.$$
	The backward homomorphism $\psi_i:\pi_1(\intr)\to \pi_1(M)$ has kernel $\mathbb{Z}_2$. 
\end{exmp}

\section{Construction of $\psi_i$}\label{sec:construction}

In this section, we always assume that $M_i$ is a sequence of closed $n$-manifolds with (\ref{eq:cond}) that Gromov-Hausdorff converges to $X$. Let $0<\epsilon<\epsilon_0(n)/2$, where $\epsilon_0(n)$ is the constant in Theorem \ref{pre:CC_reg_contract}. Let $x$ be a regular point of $X$ and $x_i$ in $M_i$ converging to $x$. By the proof of Lemma \ref{lem:reg_param}(1), there is $\delta>0$ such that $B_{\delta}(x)\subseteq \mathcal{R}_{\epsilon,\delta}$, thus $x\in \intr$. We may further shrink this $\delta$ later. The main goal of this section is to construct the group homomorphisms
$$\psi^\delta_i: \pi_1(\intr,x)\to \pi_1(M_i,x_i)$$
for all $i$ large. 

\begin{lem}\label{lem:local_contract}
	Given any $0<\epsilon<\epsilon_0(n)/2$ and $\delta>0$, the following holds for all large $i$.
	
	Let $z_i$ be a point in $M_i$ that is $\delta/30$-close to a point $z\in \intr$. Then any loop in $B_{\delta/30}(z_i)$ is contractible in $B_{\delta}(z_i)$. 
\end{lem}

\begin{proof}
	We set
	$$\eta_i=d_{GH}(M_i,X)\to 0.$$
	Then for each $z\in X$, we can choose a point $w_i\in M_i$ that is $\eta_i$-close to $z$. By the convergence $M_i\ghto X$ and the compactness of $X$, there is $i_0$ large such that
	$$d_{GH}(B_\delta(w_i),B_\delta(z))\le \dfrac{\epsilon_0(n)}{2} \delta$$
	holds for all $z\in X$, all $w_i\in M_i$ that is $\eta_i$-close to $z$, and all $i\ge i_0$, where $\epsilon_0(n)$ is the constant in Theorem \ref{pre:CC_reg_contract}.
	
	Now fixing a point $z\in\intr$, we have
	$$d_{GH}(B_\delta(z),B_{\delta}^n(0))\le\epsilon \delta.$$
	Thus by triangle inequality,
	$$d_{GH}(B_\delta(w_i),B_{\delta}^n(0))\le (\epsilon+\epsilon_0(n)/2)\delta<\epsilon_0(n)\delta.$$
	Then by Theorem \ref{pre:CC_reg_contract}, every loop in $B_{\delta/10}(w_i)$ is contractible in $B_{\delta/5}(w_i)$. Let $z_i$ be any point in $M_i$ that is $\delta/30$-close to $z$. We have
	$$d(z_i,w_i)\le d(z_i,z)+d(z,w_i)\le \delta/30+\eta_i.$$
	Thus when $i$ is large with $\eta_i<\delta/30$, we see that $B_{\delta/30}(z_i)\subseteq B_{\delta/10}(w_i)$. Therefore, every loop in $B_{\delta/30}(z_i)$ is contractible in $B_{\delta/5}(w_i)\subseteq B_{\delta}(z_i)$.
\end{proof}

With Lemma \ref{lem:local_contract}, we follow a similar construction in \cite[Lemma 2.4]{PW} (also see \cite{Tu}) to construct nearby loops and homotopies on $M_i$ from the ones on $\intr$. For two compact length metric spaces $(X_1,x_1)$ and $(X_2,x_2)$ that are close in the Gromov-Hausdorff distance, we say that two curves $\sigma_j:[0,1] \to X_j$, where $j=1,2$, are $\epsilon$-close, if 
$$d(\sigma_1(t),\sigma_2(t))\le \epsilon$$
for all $t\in [0,1]$.

\begin{lem}\label{lem:nearby_homotopy}
	We write $\eta_i=d_{GH}(M_i,X)\to 0$. Then for sufficiently large $i$, the followings hold.\\
	(1) For any loop $\sigma:[0,1]\to \intr$, there is a loop $\sigma_i$ in $M_i$ that is $5\eta_i$-close to $\sigma$.\\
	(2) Let $\sigma_i$ and $\sigma'_i$ be loops in $M_i$ that are both $\delta/300$-close to a loop $\sigma$ in $\intr$, then $\sigma_i$ and $\sigma'_i$ are free homotopic in $M_i$.\\
	(3) Let $\sigma$ and $\tau$ be two loops in $\intr$. Let $\sigma_i$ and $\tau_i$ be loops in $M_i$ that is $\delta/300$-close to $\sigma$ and $\tau$, respectively. If $\sigma$ and $\tau$ are free homotopic in $\intr$, then $\sigma_i$ and $\tau_i$ are free homotopic in $M_i$.
\end{lem}

\begin{proof}
	(1) The construction of $\sigma_i$ is the same as the proof of \cite[Lemma 2.4(1)]{PW}. Namely, using the uniform continuity of $\sigma$, we choose a suitable partition of $[0,1]$. Then for each intermediate point in the partition, we can pick nearby points in $M_i$ and then join them by minimal geodesics.
	 
	(2) By uniform continuity of $\sigma$, we choose $l>0$ such that
	$$d(\sigma(t),\sigma(t'))\le \delta/300$$
	for all $t,t'\in[0,1]$ with $|t-t'|\le l$. Let $\{t_0=0,t_1,...,t_j,...,t_N=1\}$ be a partition of $[0,1]$ with $|t_{j+1}+t_j|\le l$ for all $j$. By triangle inequality, it is clear that
	$$d(\sigma_i(t_j),\sigma_i(t_{j+1}))\le 3\cdot \delta/300,\quad d(\sigma'_i(t_j),\sigma'_i(t_{j+1}))\le 3\cdot \delta/300.$$
	Let $c_{i,j}$ be the loop obtained by joining $\sigma_i|_{[t_j,t_{j+1}]}$, a minimal geodesic from $\sigma_{i}(t_{j+1})$ to $\sigma'_i(t_{j+1})$, the inverse of $\sigma'_i|_{[t_j,t_{j+1}]}$, and lastly a minimal geodesic from $\sigma'_{i}(t_j)$ to $\sigma_i(t_j)$. Since 
	$$d(\sigma_i(t_j),\sigma'_i(t_j))\le 2\cdot \delta/300$$
	for all $i$. By construction, one can verify that 
	$$\text{image of }c_{i,j}\subseteq B_{\delta/30}(\sigma_i(t_j)).$$
	Because $\sigma_i(t_j)$ is $\delta/300$-close to $\sigma(t_j)\in \intr$, by Lemma \ref{lem:local_contract}, $c_{i,j}$ is contractible in $M_i$ for all $j$. Thus $\sigma_i$ and $\sigma'_i$ are free homotopic.
	
	(3) Let $H:S^1\times [0,1]\to \intr$ be a homotopy between $\sigma$ and $\tau$. We follow the method in \cite[Lemma 2.4]{PW} to construct a homotopy $H_i$ between $\sigma_i$ and $\tau_i$ as below. By the uniform continuity of $H$, we can choose a finite triangular decomposition $\Sigma$ of $S^1\times [0,1]$ so that $$\mathrm{diam}(H(\Delta))\le \delta/300$$
	for each triangle $\Delta$ of $\Sigma$. For any vertex $v$ of $\Sigma$, if $v$ is on the boundary of $S^1\times [0,1]$, then $H_i(v)$ is naturally defined as a point on $\sigma_i$ or $\tau_i$; if not, then we define $H_i(v)$ as a point in $M_i$ that is $\eta_i$-close to $H(v)$. Next, we define $H_i$ on every edge of $\Sigma$: for an edge that is on the boundary of $S^1\times [0,1]$, $H_i$ on this edge is naturally defined as part of $\sigma_i$ or $\tau_i$; for an edge not on the boundary with vertices $v$ and $w$, we map it to a minimal geodesic between $H_i(v)$ and $H_i(w)$. If $\eta_i\le \delta/300$, then by construction, every triangle $\Delta$ satisfies
	$$H_i(\partial \Delta)\subseteq B_{\delta/30}(H_i(v)),$$
	where $v$ is a vertex of $\Delta$. Since $H_i(v)$ is $\delta/300$-close to $H(v)\in \intr$, we can apply Lemma \ref{lem:local_contract} to contract the loop $H_i(\Delta)$. Applying this to all the triangles of $\Sigma$, we result in the desired homotopy between $\sigma_i$ and $\tau_i$.
\end{proof}

Now we construct the backward homomorphism $\psi_i^\delta$.

\begin{defn}\label{defn:backward}
Let $[\sigma]\in \pi_1(\intr,x)$ represented by a loop $\sigma$ based at $x$ in $\intr$. For $i$ large that fulfills Lemma \ref{lem:nearby_homotopy}, we draw a loop $\sigma_i$ in $M_i$ based at $x_i$ that is $\delta/300$-close to $\sigma$. We define
\begin{align*}
\psi^\delta_i:\pi_1(\intr,x)&\to \pi_1(M_i,x_i),\\
 [\sigma]&\mapsto [\sigma_i].
\end{align*}
\end{defn}

\begin{thm}\label{thm:well_defined}
	The above constructed $\psi_i^\delta$ is well-defined and is a group homomorphism for all $i$ large.
\end{thm}

\begin{proof}
	By Lemma \ref{lem:nearby_homotopy}(2), $\psi_i^\delta[\sigma]=[\sigma_i]$ is independent of the choice of $\sigma_i$. It also follows from Lemma \ref{lem:nearby_homotopy}(3) that the definition is independent of the choice of $\sigma$. 
	
	It is straightforward to check that $\psi_i^\delta$ is a group homomorphism. In fact, let $\sigma$ and $\tau$ be two loops in $\intr$ based at $x$, and let $\sigma_i$ and $\tau_i$ be loops in $M_i$ that is $\delta/300$-close to $\sigma$ and $\tau$, respectively. Since the the product $\sigma_i\cdot\tau_i$ is  clearly $\delta/300$-close to $\sigma\cdot\tau$, by definition, we have
	$$\psi_i^\delta[\sigma]\cdot \psi_i^\delta[\tau]=[\sigma_i]\cdot [\tau_i]=[\sigma_i\cdot\tau_i]=\psi_i^\delta[\sigma\cdot \tau]=\psi_i^\delta([\sigma]\cdot[\tau]).$$
\end{proof}

For $0<\epsilon\le \epsilon'$ and $0<\delta'\le \delta$, we have inclusion 
$$\intr \subseteq \intrp.$$
For both $\intr$ and $\intrp$, we have backward homomorphisms defined; they are indeed related by the inclusion map, as stated in Lemma \ref{lem:comp_incl} below. Due to the dependence on $\epsilon$, we will write $\psi^{\epsilon,\delta}_i$ instead of $\psi_i^\delta$ for clarity.

\begin{lem}\label{lem:reg_inclusion}
	Let $0<\epsilon\le \epsilon'<\epsilon_0(n)/2$ and $0<\delta'\le \delta$. Suppose that $i$ is large such that both homomorphisms
	$$\psi_i^{\epsilon,\delta}: \pi_1(\intr,x)\to \pi_1(M_i,x_i),\quad \psi_i^{\epsilon',\delta'}: \pi_1(\intrp,x)\to \pi_1(M_i,x_i)$$
	are defined. Then $\psi_i^{\epsilon,\delta}$ coincides with the composition
	$$\pi_1(\intr,x)\overset{\iota_\star}\longrightarrow\pi_1(\intrp,x)\overset{\psi_i^{\epsilon',\delta'}}\longrightarrow \pi_1(M_i,x_i),$$
	where $\iota$ is the inclusion map $\intr\hookrightarrow\intrp$.
\end{lem}

\begin{proof}
	Let $[\sigma]\in \pi_1(\intr,x)$, where $\sigma$ is a loop in $\intr$ based at $x$. Then $\iota\circ \sigma$ naturally represents an element of $\pi_1(\intrp,x)$. Let $\sigma_i$ be a loop in $M_i$ based at $x_i$ that is $\delta'/300$-close to $\iota\circ \sigma$. According to Definition \ref{defn:backward}, we have
	$$\psi_i^{\epsilon',\delta'} \circ \iota_\star [\sigma]=\psi_i^{\epsilon',\delta'}[\iota\circ \sigma]=[\sigma_i].$$
	Since $\delta'\le \delta$, the loop $\sigma_i$ is also $\delta/300$-close to $\iota\circ\sigma=\sigma$ in $\intr$. Therefore,
	$$\psi_i^{\epsilon,\delta}[\sigma]=[\sigma_i]=\psi_i^{\epsilon',\delta'} \circ \iota_\star [\sigma].$$
\end{proof}

\section{Surjectivity of $\psi_i$}\label{sec:surj}

The main goal of this section is to prove Theorem \ref{thm:backward_surj}. The proof of surjectivity of $\psi_i^\delta$ is a contradicting argument and we shall apply equivariant GH convergence to the contradicting sequence.

Before starting the proof of Theorem \ref{thm:backward_surj}, we prove some results about the equivariant GH convergence.

\begin{lem}\label{lem:length_bound}
	Let us consider the diagram (\ref{CD:cover}) with conditions (\ref{eq:local_cond}). Suppose that $x\in \intr$, where $0<\epsilon\le \epsilon(n)$ and $0<\delta\le \delta(n)$ are sufficiently small. Then there is a constant $l(n,\delta)>0$ such that any nontrivial element in $\pi_1(M_i,x_i)$ has length at least $l(n,\delta)$, where $i$ is large.
\end{lem}	

\begin{proof}
	The proof is a localized version of an argument by Anderson \cite{An}.
	
	Let $g_i\in \pi_1(M_i,x_i)$ with $d(g_i\tilde{x}_i,\tilde{x}_i)=l_i>0$. We shall prove a lower bound for $\liminf l_i:=l$. Let $F_i$ be the Dirichlet domain of $\widetilde{M}_i$ centered at $\tilde{x}_i$. Since 
	$$g_i(F_i\cap B_{\delta}(\tilde{x}_i))\subseteq B_{l_i+\delta}(\tilde{x}_i), \quad g_i(F_i\cap B_{\delta}(\tilde{x}_i)) \cap (F_i\cap B_{\delta}(\tilde{x}_i))=\emptyset,$$
	we have volume estimate
	\begin{align*}
		2 \mathrm{vol}(B_{\delta}(x_i))&= \mathrm{vol}(F_i\cap B_{\delta}(\tilde{x}_i))+\mathrm{vol}(g_i(F_i\cap B_{\delta}(\tilde{x}_i)))\\
		&\le \mathrm{vol}(B_{l_i+\delta}(\tilde{x}_i))\\
		&\le v(n,-1,l_i+\delta),
	\end{align*}
	where $v(n,-\kappa,r)$ means the volume of an $r$-ball in the $n$-dimensional space form of constant curvature $-\kappa$. By volume convergence, as $i\to \infty$, we have
	\begin{align*}
	\mathrm{vol}(B_{\delta}(x_i))&\to \mathcal{H}^n(B_{\delta}(x))\\
	&\ge (1-\Psi(\epsilon|n))\cdot v(n,0,\delta)\\
	&\ge (1-\Psi(\epsilon|n))\cdot (1-\Psi(\delta|n)) \cdot v(n,-1,\delta).
	\end{align*}
	These lead to
	$$\dfrac{v(n,-1,l_i+\delta)}{v(n,-1,\delta)} \ge 1.9(1-\Psi(\epsilon,\delta|n))>1.5.$$
	for all $i$ large, which gives a universal lower bound $l(n,\delta)$ for $\liminf l_i$.
\end{proof}

\begin{cor}\label{cor:discrete_limit}
	In the diagram (\ref{CD:cover}) with conditions (\ref{eq:local_cond}), the limit group $\Gamma$ is discrete.
\end{cor}

\begin{proof}
	Let $z\in X$ be a regular point. We choose small $0<\epsilon<\epsilon(n)$ and $\delta>0$ such that $z\in\intr$. Let $z_i\in M_i$ converging to $z$ and let $\tilde{z}_i\in \widetilde{M}_i$ be a lift of $z_i$. By Lemma \ref{lem:length_bound}, the orbit $\Gamma_i \cdot \tilde{z}_i$ is $l(n,\delta)$-discrete. Passing this to the limit, we see that $\Gamma$ is a discrete group. 
\end{proof}

\begin{lem}\label{lem:order}
	Let $(N_i,x_i)\in\mathcal{M}(n,-1,v)$ with an isometric $\Gamma_i$-action on each $N_i$. Suppose that the sequence converges
	$$(N_i,x_i,\Gamma_i)\ghto (Y,y,G)$$
	and the limit group $G$ is discrete.
	Let $g\in G$ be an element of finite order $k$ and let $\gamma_i\in \Gamma_i$ converging to $g$. Then\\
	(1) $\gamma_i$ has order $k$ for all $i$ large;\\
	(2) $\langle\gamma_i\rangle \ghto \langle g \rangle$, where $\langle\cdot\rangle$ means the subgroup generated by that element.
\end{lem}

\begin{proof}
	(1) First note that $\gamma_i^k\ghto g^k=e$ as $i\to\infty$. We claim that $\langle\gamma_i^k\rangle \ghto \{e\}$. In fact, let $H$ be the limit of $\langle\gamma_i^k\rangle$ and suppose that $H$ has a non-identity element $h$. We pick a point $z\in Y$ with $d(hz,z)>0$. Since $d(\gamma_i^k z_i,z_i)\to 0$, where $z_i\in M_i$ converging to $z$, for any $0<l<d(hz,z)$, we can find a sequence $m_i$ such that 
	$$d((\gamma_i^k)^{m_i}z_i,z_i)\to l.$$ 
	The sequence $(\gamma_i^k)^{m_i}$ would converge to an element of $H$ with displacement $l$ at $z$. Because $l\in (0,d(hz,z))$ is arbitrary, we result in a contradiction to the discreteness of $G$. This proves the claim.
	
	By this claim, we have $D_{1,x_i}(\langle\gamma_i^k\rangle)\to 0$. On the other hand, by Theorem \ref{pre:nss} $$D_{1,x_i}(\langle\gamma_i^k\rangle)\ge\delta(n,v)>0$$
	if $\langle\gamma_i^k\rangle$ is nontrivial. We conclude that $\gamma_i^{k}=e$. It is clear that $\gamma_i$ cannot have order $m$ strictly less than $k$; otherwise $\gamma_i^m\ghto e\not= g^m$. We complete the proof that $\gamma_i$ has order $k$.
	
	(2) is a direct consequence of (1). 
\end{proof}

Let $\gamma$ be an isometry of $Y$. We write
$$\mathrm{Fix}(\gamma)=\{ z\in Y | \gamma z=z \}$$
as the fixed point set of $\gamma$.

\begin{prop}\label{thm:fix_codim_2}
	In the convergence (\ref{CD:cover}) with conditions (\ref{eq:cond}), $\mathrm{Fix}(\gamma)$ has Hausdorff dimension at most $n-2$ for all non-identity $\gamma\in \Gamma$.
\end{prop}

We provide two proofs of Proposition \ref{thm:fix_codim_2}. As we shall soon see from the second proof, which relies on Theorem \ref{pre:no_fix_sing}, a stronger result $\mathrm{Fix}(\gamma)\subseteq \mathcal{S}$ holds.

\begin{proof}[Proof I of Proposition \ref{thm:fix_codim_2}]
	Suppose the contrary $\dimH(\mathrm{Fix}(\gamma))> n-2$. Then $\mathcal{H}^{l}(\mathrm{Fix}(\gamma))>0$ for some real number $n-2<l<n$. Let $\mathcal{S}$ be the singular set of $Y$. By Theorem \ref{pre:CC_sing_dim}, $\mathcal{C}:=\mathrm{Fix}(\gamma)-\mathcal{S}$ also satisfies $\mathcal{H}^{l}(\mathcal{C})>0$. Let $z$ be an $l$-density point of $\mathcal{C}$, that is, $z\in \mathcal{R}\cap \mathrm{Fix}(\gamma)$ such that
	$$\limsup_{r\to 0}\dfrac{\mathcal{H}^l_\infty(\mathcal{C}\cap B_r(y))}{\omega_lr^l}\ge 2^{-l}.$$
	Let $r_j\to\infty$ be a sequence that realizes the above limsup and let 
	$$(r_j Y,z)\ghto (C_zY=\mathbb{R}^n,v)$$
	be a corresponding tangent cone at $z$. With respect to this convergent sequence, $\gamma$ subconverges to a limit isometry $g$ of $\mathbb{R}^n$, and $\mathcal{C}$ subconverges to a closed subset $\mathcal{C}_z \subseteq \mathbb{R}^n$. It is clear that by construction, $g$ fixes every point in $\mathcal{C}_z$. By a standard covering argument, $\mathcal{C}_y\cap B_1(v)$, and thus $\mathrm{Fix}(g)\cap B_1(v)$, have positive $l$-dimensional Hausdorff measure. 
	
	Since $g$, which an isometry of $\mathbb{R}^n$, satisfies $\dimH(\mathrm{Fix}(g))>n-2$, we conclude that $g$ must be a reflection of $\mathbb{R}^n$ that fixes a hyperplane. In particular, $g$ has order $2$. By Lemma \ref{lem:order}(1), $\gamma$ has order $2$ as well. Let $\gamma_i\in \Gamma_i$ that converges to $\gamma$ and let $\overline{M_i}=\widetilde{M_i}/\langle\gamma_i\rangle$. Lemma \ref{lem:order} allows us to consider the convergence
	\begin{equation}\label{CD:cover_tangent}
		\begin{CD}
			(\widetilde{M}_i,{z}_i,\langle\gamma_i\rangle) @>GH>> 
			(Y,z,\langle\gamma\rangle)\\
			@VV\pi_i V @VV\pi V\\
			(\overline{M_i},\bar{z}_i) @>GH>> (\overline{Y}=Y/\langle\gamma\rangle,\bar{z}),
		\end{CD}\quad \begin{CD}
		(r_jY,z,\langle\gamma\rangle) @>GH>> 
		(\mathbb{R}^n,v,\langle g\rangle)\\
		@VV\pi_i V @VV\pi V\\
		(r_j\overline{Y},\bar{z}) @>GH>> (\mathbb{R}^n/\langle g \rangle,\bar{v}).
		\end{CD}
	\end{equation}
    Because $g$ is a reflection in $\mathbb{R}^n$, the quotient $\mathbb{R}^n/\langle g \rangle$ is isometric to the Euclidean halfspace $\mathbb{H}^n=\{(a_1,...,a_n)|a_n\ge 0\}$. In particular, $\mathbb{H}^n$ appears as a tangent cone of a non-collapsing Ricci limit space $\overline{Y}$ at $\bar{y}$. This is a contradiction to Theorem \ref{pre:CC_sing_dim} and thus completes the proof.
\end{proof}

\begin{proof}[Proof II of Theorem \ref{thm:fix_codim_2}]
	The second proof is much shorter thanks to a result by Chen-Rong-Xu \cite{CRX}, that is, Theorem \ref{pre:no_fix_sing} which we have recalled in Section \ref{sec:pre}. We shall show that $\mathrm{Fix}(\gamma) \subseteq \mathcal{S}$; then the Hausdorff dimension estimate follows from Theorem \ref{pre:CC_sing_dim}. 
	
	In fact, let $z\in Y$ be a regular point and let $0<\epsilon<\epsilon(n,v)$, the constant in Theorem \ref{pre:no_fix_sing}. Then there is some $\delta>0$ such that $z$ is $(\epsilon,\delta)$-regular. Applying Theorem \ref{pre:no_fix_sing} to the first diagram of (\ref{CD:cover_tangent}), we conclude that $\langle \gamma \rangle$-action, and thus $\gamma$, does not fix $z$.
\end{proof}
	
We are in a position to prove Theorem \ref{thm:backward_surj}. For reader's convenience, we restate the surjectivity part in Theorem \ref{thm:backward_surj} as below.
	
\begin{thm}\label{thm:surj}
	Let $\psi^\delta_i: \pi_1(\intr,x)\to \pi_1(M_i,x_i)$ be the group homomorphism constructed in Definition \ref{defn:backward}. When $\delta$ is sufficiently small, $\psi^\delta_i$ is surjective for all $i$ large.
\end{thm}	

\begin{proof}
	We argue by contradiction. Suppose that for each $1/j$, where $j\in\mathbb{N}$, we can find some $i(j)\ge i$ and some element $g_{i(j)}\in \pi_1(M_{i(j)},x_{i(j)})$ such that $g_{i(j)}$ is not in the image of $\psi_{i(j)}^{1/j}$. Since $\mathrm{diam}(M_i)\le D$, $\pi_1(M_i,x_i)$ can be generated by elements of length at most $2D$. Together with Lemma \ref{lem:length_bound}, without loss of generality, we will assume that each $g_{i(j)}$ has length between $l(n,\delta)$ and $2D$ at $x_{i(j)}$.
	
	For this sequence $i(j)$, after passing to a subsequence if necessary, we consider the equivariant Gromov-Hausdorff convergence:
	\begin{center}
		$\begin{CD}
			(\widetilde{M}_{i(j)},\tilde{x}_{i(j)},\Gamma_{i(j)},g_{i(j)}) @>GH>> 
			(Y,y,\Gamma,g)\\
			@VV\pi_i V @VV\pi V\\
			(M_{i(j)},x_{i(j)}) @>GH>> (X,x).
		\end{CD}$
	\end{center}
    Because $x$ is regular, so is $y$. Under the isometry $g$, $gy$ is regular as well with $d(gy,y)\in [l(n,\delta),2D]$. By proof II of Theorem \ref{thm:fix_codim_2}, the points $g$ and $gy$ are not fixed by any $\gamma\in \Gamma-\{e\}$. Let
    $$\mathcal{C}=\mathcal{S}(Y) \cup \left(  \bigcup_{\gamma\in \Gamma-\{e\}} \mathrm{Fix}(\gamma) \right).$$
    (In fact, one can set $\mathcal{C}=\mathcal{S}(Y)$ because the fixed point set is contained in the singular set, as we have seen in the second proof of Theorem \ref{thm:fix_codim_2}.) By Theorem \ref{pre:CC_sing_dim}, $\mathcal{C}$ has Hausdorff dimension at most $n-2$. We note that $y,gy\in Y-\mathcal{C}$ because they are in $\pi^{-1}(x)$ and thus not fixed by any $\gamma\in\Gamma-\{e\}$ according to Lemma \ref{lem:length_bound}. As a result of Theorem \ref{pre:CC_cnt}, we can connect $y$ and $gy$ by a path $\sigma$ that is contained in $Y-\mathcal{C}$. In particular, $\sigma$ is in the regular set and avoids any point that is fixed by some nontrivial element of $\Gamma$. 
     Let $\delta_1>0$ be the distance between $\sigma$ and $\cup_{\gamma\in \Gamma-\{e\}}\mathrm{Fix}(\gamma)$ and let 
     $$T:=T_{\delta_1/2}(\sigma)=\{z\in Y \ |\ d(z,\sigma)\le \delta_1/2 \}$$ 
     be the closed tubular neighborhood of $\sigma$ with radius $\delta_1/2$. By construction, $T$ does not intersect $\mathrm{Fix}(\gamma)$ for all non-identity $\gamma\in \Gamma$. Because $T$ is compact, 
     $$\delta_2:=\inf_{a\in T,\gamma\in\Gamma-\{e\}} d(a,\gamma a)$$
     is positive. 
     
     Setting $\delta_3=\min\{\delta_1/2,\delta_2/4\}$, we claim that $B_{\delta_3}(z)$ is isometric to $B_{\delta_3}(\pi(z))\subseteq X$ for all $z\in \sigma$, where $\pi: Y\to X=Y/\Gamma$ is the quotient map. In fact, first note that for any two points $a,b\in B_{\delta_3}(z)$, we clearly have $a,b\in T$. Then for any other orbit point $a'\in \Gamma a-\{a\}$, it follows from triangle inequality that 
     $$d(a',b)\ge d(a',a)-d(a,b)\ge \delta_2-2\delta_3\ge \delta_2/2 > d(a,b).$$
     This verifies the claim: for all $a,b\in B_{\delta_3}(z)$, $$d_Y(a,b)=d_Y(\Gamma a, \Gamma b)=d_X(\pi(a),\pi(b)).$$
     
     We choose a small $\epsilon_1>0$ such that $\Psi(\epsilon_1|n)\le \epsilon$, where $\Psi$ is the function in Lemma \ref{lem:reg_param}(1). With this $\epsilon_1$, by Lemma \ref{lem:reg_param}(2), there is $\delta_4>0$ such that 
     $$\sigma \subseteq \mathcal{R}_{\epsilon_1,\delta_4}(Y).$$
     Let $\delta_5:=\min\{\delta_3,\delta_4\}>0$. Since $B_{\delta_5}(z)$ is isometric to $B_{\delta_5}(\pi(z))$ for all $z\in\sigma$, together with Lemma \ref{lem:reg_param}(1), we conclude that
     $$\pi(\sigma) \subseteq \mathcal{R}_{\epsilon_1,\delta_5}(X) \subseteq \mathrm{Int}\mathcal{R}_{\epsilon,\delta_5/3}.$$
     
     Now we go back to the sequence of manifolds. Along $\widetilde{M}_{i(j)}$, let $\sigma_{i(j)}$ be a sequence of paths from $\tilde{x}_{i(j)}$ to $g_{i(j)}\tilde{x}_{i(j)}$ that converges uniformly to $\sigma$. Then its projection $\pi_{i(j)}(\sigma_{i(j)})=:\overline{\sigma_{i(j)}}$ is a loop that represents $g_{i(j)}$ and uniformly converges to a loop $\pi(\sigma)$ in $X$ as $j\to\infty$. By the construction in Definition \ref{defn:backward}, when $j$ is large we have
     $$\psi_{i(j)}^{\delta_5/3}: \pi_1(\mathrm{Int}\mathcal{R}_{\epsilon,\delta_5/3},x)\to \pi_1(M_i,x_i)\  \text{  with  }\  \psi_{i(j)}^{\delta_5/3}[\overline{\sigma_{i(j)}}]=g_{i(j)}.$$
     Applying Lemma \ref{lem:reg_inclusion} with $\epsilon=\epsilon'$, we obtain
     $$g_{i(j)}=\psi_{i(j)}^{\delta_5/3}[\overline{\sigma_{i(j)}}]= \psi_{i(j)}^{1/j} \circ \iota_\star [\overline{\sigma_{i(j)}}],$$
     where $\iota$ is the inclusion map $\mathrm{Int}\mathcal{R}_{\epsilon,\delta_5/3} \hookrightarrow \mathrm{Int}\mathcal{R}_{\epsilon,1/j}$.
     In particular, $g_{i(j)}$ is in the image of $\psi_{i(j)}^{1/j}$. This contradicts with our choice in the beginning that $g_{i(j)}$ is not in the image of $\psi_{i(j)}^{1/j}$ and thus completes the proof.
\end{proof}

With Theorems \ref{thm:well_defined} and \ref{thm:surj}, now we complete the proof of Theorem \ref{thm:backward_surj} by Lemma \ref{lem:comp_incl} below.

\begin{lem}\label{lem:comp_incl}
	Let
	$$\phi_i: \pi_1(M_i,x_i)\to \pi_1(X, x),\quad \psi^\delta_i: \pi_1(\intr,x)\to \pi_1(M_i,x_i)$$
	be the surjective homomorphisms in Theorems \ref{thm:forward_surj} and \ref{thm:surj}, respectively. Then
	$$\phi_i \circ \psi^\delta_i: \pi_1(\intr,x)\to \pi_1(X,x)$$
	coincides with $\iota_\star$ for all $i$ large, where $\iota:\intr \hookrightarrow X$ is the inclusion map.
\end{lem}

\begin{proof}
	Because $X$ is semi-locally simply connected \cite{PW}, there is $\delta_0>0$ such that every loop contained in a $\delta_0$-ball of $X$ is contractible in $X$. We set 
	$$\delta_1=\min\{ \delta_0/20, \delta/300 \}.$$
	
	We recall that the forward homomorphism $\phi_i$ can be constructed as follows (see \cite{Tu} or \cite{PW} for details). When $i$ is large such that $d_{GH}(M_i,X)\le \delta_1$, for any loop $\sigma_i$ in $M_i$ based at $x_i$, we can draw a loop $\sigma$ in $X$ based at $x$ such that $\sigma$ is $5\delta_1$-close to $\sigma_i$. Then one can define the desired $\phi_i$ by sending $[\sigma_i]$ to $[\sigma]$. The choice of $\delta_0$ assures that $\phi_i$ is well-defined and a surjective homomorphism.
	
	Now let $[\sigma]\in \pi_1(\intr,x)$ represented by a loop $\sigma$ based at $x$ in $\intr$. When $i$ is large, let $\sigma_i$ be a loop based at $x_i\in M_i$ that is $\delta_1$-close to $\sigma$. By the constructions of $\phi_i$ and $\psi_i^\delta$, we have
	$$\phi_i \circ \psi^\delta_i[\sigma]=\phi_i[\sigma_i]=[\sigma]\in \pi_1(X,x).$$
\end{proof}
	 

Next, we prove Theorem \ref{thm:reg_inj_isom}.

\begin{proof}[Proof of Theorem \ref{thm:reg_inj_isom}]
	We choose a sufficiently small $\delta>0$ so that we can apply Theorem \ref{thm:backward_surj} to construct surjective group homomorphisms
	$$\psi_i^\delta:\pi_1(\intr,x)\to \pi_1(M_i,x_i)$$
	for all $i$ large. If the inclusion map $\iota^\delta: \intr \hookrightarrow X$ induces an injective homomorphism 
	$$\iota^\delta_\star : \pi_1(\intr,x)\to \pi_1(X,x),$$
	then by Theorem \ref{thm:backward_surj}(2), the composition
	$$\pi_1(\intr,x) \overset{\psi_i^\delta}\longrightarrow \pi_1(M_i,x_i) \overset{\phi_i}\longrightarrow \pi_1(X,x)$$
	is an isomorphism. Together with the surjectivity of $\psi_i^\delta$ and $\phi_i$, we clearly have isomorphism $\pi_1(X)\simeq \pi_1(M_i)$.
	
   In general, if $\iota^\delta_\star$ is not injective, we shall analyze its kernel. We claim that
   $$\ker \iota^\delta_\star = \ker \psi_i^\delta.$$
   If this claim holds, then 
   $$ \pi_1(M_i,x_i) = \dfrac{\pi_1(\intr,x)}{\ker \psi_i^\delta}=\dfrac{\pi_1(\intr,x)}{\ker \iota^\delta_\star}=\pi_1(X,x).$$
   One side of the inclusion $\ker \psi_i^\delta \subseteq \ker \iota^\delta_\star$ is clear due to Theorem \ref{thm:backward_surj}(2). It remains to prove the other direction.
   
   Let us consider a composition of inclusion maps $\iota \circ j= \iota^\delta$:
   $$\intr \overset{j} \hookrightarrow \intrr \overset{\iota}\hookrightarrow X.$$
   They induce
   $$\pi_1(\intr,x)\overset{j_\star} \longrightarrow \pi_1(\intrr,x) \overset{\iota_\star}\longrightarrow \pi_1(X,x)$$
   with $\iota_\star \circ j_\star = \iota^\delta_\star$ being surjective.
   By the assumption that $\iota_\star$ is injective, $\iota_\star$ is an isomorphism and
   $$\ker \iota^\delta_\star = \ker j_\star.$$
   Let $[\sigma]\in \ker j_\star$ represented by a loop $\sigma$ at $x$ in $\intr$. Then $\sigma$ is contractible in $\intrr$. Let $H:[0,1]^2\to \intrr$ be a nullhomotopy of $\sigma$. By Lemma \ref{lem:reg_param}(1,2), there are $\epsilon'=\Psi(\epsilon|n)>\epsilon$ and $0<\delta'<\delta$ such that 
   $$H([0,1]^2) \subseteq \intrp.$$
   When $i$ is large, we draw a loop $\sigma_i$ based at $x_i\in M_i$ that is $\delta'/300$-close to $\sigma$. It follows from Lemma \ref{lem:nearby_homotopy}(3) that $\sigma_i$ is contractible in $M_i$. By the construction of 
   $$\psi_i^{\epsilon',\delta'} : \pi_1(\intrp,x)\to \pi_1(M_i,x_i)$$
   and Lemma \ref{lem:comp_incl}, we have
   $$\psi_i^{\epsilon,\delta}[\sigma]=\psi_i^{\epsilon',\delta'}[\sigma]=[\sigma_i]=\mathrm{id} \in \pi_1(M_i,x_i).$$
   This shows that
   $$\ker \psi_i^{\epsilon,\delta} \supseteq \ker j_\star=\ker \iota^\delta_\star$$
   and hence completes the proof.
\end{proof}

\end{document}